\documentclass[a4paper,11pt] {amsart}
\usepackage{amsthm,amssymb,latexsym,mathrsfs}
\input amssym.def

\author{Beno\^it F. Sehba}

\title[Inequalities for the Fractional Bergman operators ]{Weighted norm inequalities for fractional Bergman operators}
\newtheorem{theorem}{T{\hskip 0pt\footnotesize\bf HEOREM}}[section]
\newtheorem{lemma}[theorem]{L{\hskip 0pt\footnotesize\bf EMMA}}
\newtheorem{proposition}[theorem]{P{\hskip 0pt\footnotesize\bf ROPOSITION}}

\newtheorem{corollary}[theorem]{C{\hskip 0pt\footnotesize\bf OROLLARY}}
\newtheorem{remark}[theorem]{R{\hskip 0pt\footnotesize\bf EMARK}}

\newcommand{\bprop} {\begin{proposition}}
\newcommand{\eprop} {\end{proposition}}
\newcommand{\btheo} {\begin{theorem}}
\newcommand{\etheo} {\end{theorem}}
\newcommand{\blem} {\begin{lemma}}
\newcommand{\elem} {\end{lemma}}
\newcommand{\bcor} {\begin{corollary}}
\newcommand{\ecor} {\end{corollary}}

\newcommand{\Be}{\begin{equation}}
\newcommand{\Ee}{\end{equation}}
\newcommand{\Bea}{\begin{eqnarray}}
\newcommand{\Eea}{\end{eqnarray}}
\newcommand{\Bes}{\begin{equation*}}
\newcommand{\Ees}{\end{equation*}}
\newcommand{\Beas}{\begin{eqnarray*}}
\newcommand{\Eeas}{\end{eqnarray*}}
\newcommand{\Ba}{\begin{array}}
\newcommand{\Ea}{\end{array}}

\begin{document}
\address{Beno\^it F. Sehba, Department of Mathematics, University of Ghana, Legon, P. O. Box LG 62, Legon, Accra, Ghana}
\email{bfsehba@ug.edu.gh}
%\thanks{The second author was partially supported by the ANR project ANR-09-BLAN-0058-01}

%\noindent  ABSTRACT.\,

\maketitle
\begin{abstract}
We prove in this note one weight norm inequalities for some positive Bergman-type operators. 
\end{abstract}
\section{Introduction and results}
The set $\mathbb{R}_+^2:= \{z=x+iy\in \mathbb {C}: x\in \mathbb R,\,\,\, \textrm{and}\,\,\,y>0\}$ is called the upper-half plane. By a weight, we will mean a nonnegative locally integrable function on $\mathbb{R}_+^2$. Let $\alpha>-1$, and $1\le p<\infty$. For $\omega$ a weight, we denote by $L^p(\mathbb{R}_+^2, \omega dV_\alpha)$, the set of all functions $f$ defined on $\mathbb{R}_+^2$ such that
$$||f||_{p,\omega, \alpha }^p:=\int_{\mathbb{R}_+^2}|f(z)|^p\omega(z)dV_\alpha(z)<\infty$$
with $dV_\alpha(x+iy)=y^\alpha dxdy$. When $\omega=1$, we simply write $L^p(\mathbb{R}_+^2, dV_\alpha)$ and $\|\cdot\|_{p,\alpha}$ for the corresponding norm.
\vskip .1cm
For $\alpha>-1$ and $0\le \gamma<2+\alpha$, the positive fractional Bergman operator $T_{\alpha,\gamma}$ is defined by
\Be\label{eq:fracBergdef}
T_{\alpha,\gamma}f(z):=\int_{\mathbb{R}_+^2}\frac{f(w)}{|z-\overline{w}|^{2+\alpha-\gamma}}dV_\alpha(w).
\Ee
For $\gamma=0$, the operator $P_\alpha^+:=T_{\alpha, 0}$ is the positive Bergman projection. 
\vskip .3cm
 The above  operator can be seen as the upper-half plane analogue of the fractional integral operator (Riesz potential) defined  by
$$I_\alpha f(x)=\int_{\mathbb{R}^n}\frac{f(y)}{|x-y|^{n-\alpha}},\,\,\,x\in \mathbb{R}^n$$
for $0\leq \alpha<n$, $n\in \mathbb{N}$.
 We recall that the weighted boundedness of the latter was obtained by B. Muckenhoupt and R. L. Wheeden \cite{Wheeden}. More precisely, let $\mathcal{Q}$ be set of all cubes in $\mathbb{R}^n$. Let $1\le p,q<\infty$. We say a weight $\omega$ belongs to the class $A_{p,q}$ with $p\neq 1$, if 
$$[\omega]_{p,q}:=\sup_{Q\in \mathcal{Q}}\left(\frac{1}{|Q|}\int_{Q}\omega(x)^qdx\right)^{\frac{1}{q}}\left(\frac{1}{|Q|}\int_{Q}\omega(x)^{-p'}dx\right)^{\frac{1}{p'}}<\infty.$$
When $p=1$, we denote by $A_{1,q}$, the class of all weights $\omega$ such that 
$$[\omega]_{1,q}:=\sup_{Q\in \mathcal{Q}}\mbox{ess}\sup_{x\in Q}\left(\frac{1}{|Q|}\int_{Q}\omega(t)^qdt\right)^{\frac{1}{q}}\omega(x)^{-1}<\infty.$$ 
The results obtained by  B. Muckenhoupt and R. L. Wheeden \cite{Wheeden} summarize as follows.
\begin{theorem}\label{thm:wheeden}
Let $0<\alpha<n$. Then the following are satisfied.
\begin{itemize}
\item[(a)] Let $q=\frac{n}{n-\alpha}$. If $\omega\in A_{1,q}$, then there is a constant $C=C(\alpha,[\omega]_{1,q})$ such that 
\begin{equation}\label{eq01}\sup_{\lambda>0}\lambda\omega^q\left( \left\{x\in\mathbb{R}^n:|I_{\alpha}f(x)|>\lambda\right\}\right)^{1/q}\le C\int_{\mathbb{R}^n}|f(x)|\omega(x)dx.
\end{equation}
\item[(b)] Given $1<p<\frac{n}{\alpha}$, let $q$ be such that $\frac 1p-\frac 1q=\frac{\alpha}{n}$. If $\omega\in A_{p,q}$, then there is a constant $C=C(\alpha,[\omega]_{p,q})$ such that
\begin{equation}\label{eq2}
\left(\int_{\mathbb{R}^n}(\omega(x)|I_{\alpha}f(x)|)^q\mathrm{d}x\right)^{\frac{1}{q}}\le C\left(\int_{\mathbb{R}^n}(\omega(x)|f(x)|)^p\mathrm{d}x\right)^{\frac{1}{p}}.
\end{equation}

\end{itemize}

\end{theorem}
 Our aim in this note is to provide corresponding results for the operator $T_{\alpha,\gamma}$. For this and to present our results, we need some other definitions.
\vskip .3cm

For any interval $I\subset \mathbb{R}$, we denote by $Q_I$ its associated Carleson square, that is the set $$Q_I:=\{z=x+iy\in \mathbb C:x\in I\,\,\,\textrm{and}\,\,\,0<y<|I|\}.$$

Let $\alpha>-1$, and $1<p<\infty$. Given a weight $\omega$, we say $\omega$ is in the B\'ekoll\`e-Bonami class $\mathcal{B}_{p,\alpha}$,  if the quantity
$$[\omega]_{\mathcal{B}_{p,\alpha}}:=\sup_{I\in \mathcal{I}}\left(\frac{1}{|Q_I|_\alpha}\int_{Q_I}\omega(z)dV_\alpha(z)\right)\left(\frac{1}{|Q_I|_\alpha}\int_{Q_I}\omega(z)^{1-p'}dV_\alpha(z)\right)^{p-1}$$
%$Q_I:=\{z=x+iy\in \mathbb C:x\in I\,\,\,\textrm{and}\,\,\,0<y<|I|\}$, $|Q_I|=\int_{Q_I}dV(z)$, $pp'=p+p'$. 
is finite. Here and all over the text, for a measurable set $E\subset \mathbb{R}_+^2$, $|E|_{\omega,\alpha}:=\int_{E}\omega(z)dV_\alpha(z)$ and we write $|E|_\alpha$ for $|E|_{1,\alpha}$. Also, we have used $\mathcal{I}$ to denote the set of all intervals of $\mathbb{R}$.
\vskip .2cm
It is now well known that for $1<p<\infty$, the operator $P_\alpha^+$ is bounded on $L^p(\mathbb{R}_+^2, \omega dV_\alpha)$ if and only if $\omega\in \mathcal{B}_{p,\alpha}$ (see \cite{Bek,BB,PR}). For $p=1$, we say $\omega\in \mathcal{B}_{1,\alpha}$, if 
$$[\omega]_{\mathcal{B}_{1,\alpha}}:=\sup_{I\in \mathcal{I}}\mbox{ess}\sup_{z\in Q_I}\left(\frac{1}{|Q_I|_{\alpha}}\int_{Q_I}\omega dV_{\alpha}\right)\omega(z)^{-1}<\infty.$$ 
For $1\le p,q<\infty$, we introduce the two following classes of weights: we say a weight $\omega$ belongs to the set $B_{p,q,\alpha}$ with $p\neq 1$, if 
$$[\omega]_{B_{p,q,\alpha}}:=\sup_{I\in \mathcal{I}}\left(\frac{1}{|Q_I|_{\alpha}}\int_{Q_I}\omega^qdV_{\alpha}\right)^{\frac{1}{q}}\left(\frac{1}{|Q_I|_{\alpha}}\int_{Q_I}\omega^{-p'}dV_{\alpha}\right)^{\frac{1}{p'}}<\infty.$$
We say $\omega$ belongs to $B_{1,q,\alpha}$, if 
$$[\omega]_{B_{1,q,\alpha}}:=\sup_{I\in \mathcal{I}}\mbox{ess}\sup_{z\in Q_I}\left(\frac{1}{|Q_I|_{\alpha}}\int_{Q_I}\omega^qdV_{\alpha}\right)^{\frac{1}{q}}\omega(z)^{-1}<\infty.$$ 
We observe that if $r=1+\frac{q}{p'},$ then $$[\omega^q]_{\mathcal{B}_{r,\alpha}} = [\omega]_{B_{p,q,\alpha}}^q.$$
%;\qquad [\omega^q]_{B_{1,\alpha}} = [\omega]_{B_{1,q,\alpha}}$$
The above classes can be compared with the classes of weights $A_{p,q}$ introduced by B. Muckemhoupt and L. Wheeden in relation with the study of weighted norm inequality for the Riesz potential (see \cite{Wheeden}). 
\vskip .3cm
For the strong inequality, we obtain the following.
\begin{theorem}\label{thm:main1}
Let $\alpha>-1$ and $0\le\gamma < 2 + \alpha,$ and $1<p<\frac{2+\alpha}{\gamma}.$ Define $q$ by $\frac{1}{p} - \frac{1}{q} = \frac{\gamma}{2+\alpha}.$ Assume that the weight $\omega$ belongs to the class $B_{p,q,\alpha}$. Then $T_{\alpha,\gamma}$ is bounded from $L^p(\omega^p\mathrm{d}V_{\alpha})$ to $L^q(\omega^q\mathrm{d}V_{\alpha})$ if and only if $\omega\in B_{p,q,\alpha}.$ Moreover, 
\begin{equation}\label{eq2}
\left(\int_{\mathbb{R}_+^2}(\omega(z)|T_{\alpha,\gamma}f(z)|)^q\mathrm{d}V_{\alpha}\right)^{\frac{1}{q}}\le C_{\alpha,p}([\omega]_{B_{p,q,\alpha}})^{1+\frac{p'}{p}+\frac{q}{p'}}\left(\int_{\mathbb{R}_+^2}(\omega(z)|f(z)|)^p\mathrm{d}V_{\alpha}\right)^{\frac{1}{p}}.
\end{equation}
\end{theorem}
%As commented earlier, the strong boundedness of $P_\alpha^+$ follows from the one of $T_{\alpha,\gamma}$. We then have the following.
We observe if $\gamma\ge 0$, then for any positive function $f$ and any $z=x+iy\in \mathbb{R}_+^2$, $y^\gamma P_\alpha^+f(z)\le T_{\alpha,\gamma}f(z)$. It follows that given two weights $\sigma$ and $\omega$ on $\mathbb{R}_+^2$, for $1\le p\le q<\infty$, the (strong) boundedness of $T_{\alpha,\gamma}$ from $L^p(\mathbb{R}_+^2, \sigma dV_\alpha)$ to $L^q(\mathbb{R}_+^2, \omega dV_\alpha)$ implies the boundedness of $P_\alpha^+$ from  $L^p(\mathbb{R}_+^2, \sigma dV_\alpha)$ to $L^q(\mathbb{R}_+^2, \omega dV_\eta)$ where $\eta=\alpha+q\gamma $. In particular, taking $\sigma=\omega$ and observing that when $\gamma=(2+\alpha)\left(\frac 1p-\frac 1q\right)$, $\eta=(2+\alpha)(\frac{q}{p}-1)+\alpha$, we deduce the following from the above result.
\begin{corollary}\label{cor:main1}
Let $\alpha>-1$ and $0\le\gamma < 2 + \alpha,$ and $1<p<\frac{2+\alpha}{\gamma}.$ Define $q$ by $\frac{1}{p} - \frac{1}{q} = \frac{\gamma}{2+\alpha}.$ Let $\omega$ be weight on $\mathbb{R}^2_+$. Assume that the weight $\omega\in B_{p,q,\alpha}$. Then $P_{\alpha}^+$ is bounded from $L^p(\omega^pdV_{\alpha})$ to $L^q(\omega^qdV_{\eta})$, with $\eta=(2+\alpha)(\frac{q}{p}-1)+\alpha$. Moreover, 
\begin{equation}\label{eq2cor}
\left(\int_{\mathbb{R}_+^2}(\omega(z)|P_{\alpha}^+f(z)|)^q\mathrm{d}V_{\eta}\right)^{\frac{1}{q}}\le C_{\alpha,p}([\omega]_{B_{p,q,\alpha}})^{1+\frac{p'}{p}+\frac{q}{p'}}\left(\int_{\mathbb{R}_+^2}(\omega(z)|f(z)|)^p\mathrm{d}V_{\alpha}\right)^{\frac{1}{p}}.
\end{equation}
\end{corollary}
\vskip .2cm
For the limit case $p=1$, we obtain the following weak boundedness of the operator $T_{\alpha,\gamma}$.

\begin{theorem}\label{thm:main3}
Let $\alpha>-1$ and $0<\gamma<2+\alpha.$ Let $q = \frac{2+\alpha}{2+\alpha-\gamma}.$ Assume that the weight $\omega$ belongs to the class $B_{1,q,\alpha}$. Then $T_{\alpha,\gamma}$ is bounded from $L^1(\omega dV_{\alpha})$ into $L^{q,\infty}(\omega^qdV_{\alpha})$.  Moreover, \begin{equation}\label{eq1}\sup_{\lambda>0}\lambda\left| \left\{z\in\mathbb{R}^2_+:|T_{\alpha,\gamma}f(z)|>\lambda\right\}\right|_{\omega^q,{\alpha}}^{1/q}\le C(\alpha,\gamma)[\omega]_{B_{1,q,\alpha}}^{q}\int_{\mathbb{R}^2_+}|f(z)|\omega(z)dV_{\alpha}(z).
\end{equation}
\end{theorem}
It is not clear how to deduce the weak boundedness of the positive Bergman operator $P_\alpha^+$ from the one of $T_{\alpha,\gamma}$. We will then also prove the following.
\begin{theorem}\label{thm:main4}
Let $\alpha>-1$ and $0<\gamma<2+\alpha.$ Let $q = \frac{2+\alpha}{2+\alpha-\gamma}.$ Assume that the weight $\omega$ belongs to the class $B_{1,q,\alpha}$. Then $P_{\alpha}^+$ is bounded from $L^1(\omega dV_{\alpha})$ into $L^{q,\infty}(\omega^qdV_{\eta})$ with $\eta=(2+\alpha)(q-1)+\alpha$.  In this case, \begin{equation}\label{eq41}\sup_{\lambda>0}\lambda\left| \left\{z\in\mathbb{R}^2_+:|P_{\alpha}^+f(z)|>\lambda\right\}\right|_{\omega^q,{\eta}}^{1/q}\le C(\alpha,\gamma)[\omega]_{B_{1,q,\alpha}}^{2q-1}\int_{\mathbb{R}^2_+}|f(z)|\omega(z)dV_{\alpha}(z).
\end{equation}
\end{theorem}
% Even though we are not interested here in the sharp estimates of the above operators, 
 It is clear from \cite{Laceyetal, Sehba2} that our results are not sharp in terms of dependence on  $[\omega]_{B_{p,q,\alpha}}$. In particular, the power of $[\omega]_{B_{p,q,\alpha}}$ in the inequality (\ref{eq2}) is coarse. We also illustrate this fact with the following estimate for a concrete example of exponents that is inspired from \cite{Laceyetal}.
\vskip .3cm 
Put $$p_0 = \frac{2-\frac{\gamma}{2+\alpha}}{\frac{\gamma}{2+\alpha} - \left(\frac{\gamma}{2+\alpha}\right)^{2} + 1};\qquad\qquad q_0 = \frac{2-\frac{\gamma}{2+\alpha}}{1-\frac{\gamma}{2+\alpha}}.$$ Observe that 
$$
\frac{1}{q_0} = \frac{1}{p_0} - \frac{\gamma}{2+\alpha}$$ and $$\frac{q_0}{p_0'} = 1 - \frac{\gamma}{2+\alpha}.$$
Note also that $q_0<1+\frac{p_0'}{p_0}+\frac{q_0}{p_0'}$.
% \frac{1}{p_0'} = \frac{1}{q_0}\left(1 - \frac{\gamma}{2+\alpha}\right); & \hspace{3cm} \frac{q_0'}{q_0}\times \frac{2+\alpha}{2+\alpha-\gamma} = 1
%\end{array}
%$$
%$\frac{1}{q_0'}\left(1 - \frac{\gamma}{2+\alpha}\right) = \frac{1}{q_0}$
\begin{theorem}\label{thm:main2}
Let $\omega\in B_{p_0,q_0,\alpha},\,\, \alpha>-1.$ Then $T_{\alpha,\gamma}:L^{p_0}(\omega^{p_0}dV_{\alpha})\longrightarrow L^{q_0}(\omega^{q_0}dV_{\alpha})$ is bounded. Moreover, $$\|T_{\alpha,\gamma}\|_{L^{p_0}(\omega^{p_0}dV_{\alpha})\longrightarrow L^{q_0}(\omega^{q_0}dV_{\alpha})}\le C_{p_0,q_0,\alpha}[\omega]_{B_{p_0,q_o,\alpha}}^{q_0}.$$
\end{theorem}
For our proofs, we follow the now standard trend of techniques of sparse domination using dyadic grids. As observed above, our operators are clearly analogue of the Riesz potential. We note that a simplification of the proofs  of the results in Theorem \ref{thm:wheeden} was recently obtained by D. Cruz-Uribe \cite{Cruz2}. We follow here the approach in the online version of \cite{Cruz2} (the reader is advised to consult this online version and not the published one). We note that there is a natural sparse family (see for example \cite{Cruz2} for a definition of sparseness) on the upper-half plane made of  Carleson squares. 
 
 The strong inequalities are easier to prove than the weak inequalities and we only provide a proof here for the sake of the reader not used to these techniques. For the weak type estimates, we remark that one of the key arguments in the online version of \cite{Cruz2} is the reverse H\"older's inequality, a tool that is not available in our setting. To overcome this difficulty, we use a reverse doubling property satisfied by the B\'ekoll\'e-Bonami weights with a careful consideration of the involved constant. %This approach even allows us to obtain a better estimate than the one given in \cite{Cruz2}. 
 In the case of weak type estimate for the positive Bergman operator, there is a further difficulty due to the change of weight (power of the distance to the boundary). There is another cost to pay to overcome this other difficulty which is illustrated by the change of the power in the constant in (\ref{eq41}). 
 
 The question of sharp off-diagonal estimates for the Bergman projection has been partially answered in \cite{Sehba2}. Its extension to the full upper-triangle and Sawyer-type characterizations are considered in a forthcoming paper. Note also that in \cite{Sehba3}, we obtained some bump-conditions for the two-weight boundedness of the above fractional Bergman operators.
\vskip .1cm
In the next section, we recall some useful facts and results needed in our proofs. Here, we particularly point out the fact that our results essentially follow from their dyadic counterparts. In Section 3, we prove the weak type results. The strong inequalities are proved in Section 4. 
\vskip .1cm
 Given two positive quantities $A$ and $B$, the notation $A\lesssim B$ (resp. $B\lesssim A$) will mean that there is a universal constant $C>0$ such that $A\le CB$ (resp. $B\le CA$). When $A\lesssim B$ and $B\lesssim A$, we write $A\backsimeq B$. Notation $C_\alpha$ or $C(\alpha)$ means that the constant $C$ depends on the parameter $\alpha$.

\section{Preliminaries}
\subsection{Some properties of weights}

The following is an easy consequence of the H\"older's inequality (see \cite[Lemma 2.1.]{CarnotBenoit} for the case $\alpha=0$).
\begin{lemma}\label{carnot}
Let $1<p<\infty$, and $\alpha>-1$. Let $I\subset \mathbb{R}$ be an interval and denote by $T_I$ the upper half of the Carleson square $Q_I$. Assume that $\omega\in \mathcal{B}_{p,\alpha}.$ Then  \Be\label{eq:ineqweightsquare} |Q_I|_{\omega,\alpha} \le C_{p,\alpha}[\omega]_{\mathcal{B}_{p,\alpha}}|T_I|_{\omega,\alpha}\Ee
where $C_{p,\alpha}:=\max\{2,\left(\frac{2^{1+\alpha}}{2^{1+\alpha}-1}\right)^p\}$.
\end{lemma}
As a consequence of the above lemma, we obtain the following reverse doubling property.
\begin{lemma}\label{l4}
Let $1<p<\infty$, and $\alpha>-1$. Let $I\subset \mathbb{R}$ be an interval, and denote by $B_I$ the lower half of the Carleson square $Q_I$. Assume that $\omega\in B_{p,\alpha}$. Then $$\frac{|B_I|_{\omega,\alpha}}{|Q_I|_{\omega,\alpha}}\le \theta$$ where $\theta = 1 - \frac{1}{C_{p,\alpha}[\omega]_{B_{p,\alpha}}}$, with $C_{p,\alpha}$ the constant in (\ref{eq:ineqweightsquare}).
\end{lemma}
\subsection{Maximal functions and their boundedness}

Let $\alpha>-1$, $0\le \gamma<2+\alpha$, and let $\sigma$ be a weight. The weighted fractional maximal function $\mathcal{M}_{\sigma,\alpha,\gamma}$ is defined by
\begin{equation}\label{eq:maxfunctdef}
\mathcal{M}_{\sigma,\alpha,\gamma}f(z) :=  \sup_{\substack{I\subset\mathbb{R}\\z\in Q_I}}\frac{1}{|Q_I|_{\sigma ,\alpha}^{1-\frac{\gamma}{2+\alpha}}}\int_{Q_I}|f(w)|\sigma(w)dV_{\alpha}(w).
\end{equation}
When $\gamma=0$, the above operator is just the weighted Hardy-Littlewood maximal function denoted $\mathcal{M}_{\sigma,\alpha}$ and if moreover, $\sigma=1$, we simply write $\mathcal{M}_{\alpha}$. The unweighted fractional maximal function is just the operator $\mathcal{M}_{\alpha,\gamma}:=\mathcal{M}_{1,\alpha,\gamma}$.
\vskip .3cm
We consider the following system of dyadic grids,
$$\mathcal D^\beta:=\{2^j\left([0,1)+m+(-1)^j\beta\right):m\in \mathbb Z,\,\,\,j\in \mathbb Z \},\,\,\,\textrm{for}\,\,\,\beta\in \{0,1/3\}.$$
%For more on this system of dyadic grids and its applications, we refer to \cite{AlPottReg, HyPerez,Lerner,LerOmbroPerezetal,PR}. 
When $\beta=0$, we observe that $\mathcal D^0$ is the standard dyadic grid of $\mathbb R$, denoted $\mathcal D$.
\vskip .2cm
For any $\beta\in \{0,1/3\}$, we denote by $\mathcal{M}_{\sigma,\alpha,\gamma}^{d,\beta}$ the dyadic analogue of $\mathcal{M}_{\sigma,\alpha,\gamma}$, defined as in (\ref{eq:maxfunctdef}) but with the supremum taken over dyadic intervals in the grid $\mathcal D^\beta$.
\vskip .2cm
We have the following useful result.
\begin{lemma}\label{lem:weakfrac}
Let $1\le p\le q<\infty$, $\alpha>-1$ and $0\le \gamma<2+\alpha$. Let $\omega$ be a weight on $\mathbb{R}^2_+$, and let $\mu$ be a positive measure on $\mathbb{R}_+^2$. Then the following assertions are equivalent.
\begin{itemize}
\item[(a)] There is a constant $C_1>0$ such that for any $f\in L^p(\mathbb{R}^2_+, \omega dV_\alpha)$, and any $\lambda>0$,
\Be\label{eq:main21}
\mu(\{z\in \mathbb{R}^2_+: \mathcal{M}_{\alpha,\gamma}f(z)>\lambda\})\le \frac{C_1}{\lambda^q}\left(\int_{\mathbb{R}^2_+}|f(z)|^p\omega(z)dV_\alpha(z)\right)^{q/p}.
\Ee
\item[(b)] There is a constant $C_2>0$ such that for any $f\in L^p(\mathbb{R}^2_+, \omega dV_\alpha)$, for any $\beta\in \{0,\frac 13\}$, and any $\lambda>0$,
\Be\label{eq:main21}
\mu(\{z\in \mathbb{R}^2_+: \mathcal{M}_{\alpha,\gamma}^{d,\beta}f(z)>\lambda\})\le \frac{C_2}{\lambda^q}\left(\int_{\mathbb{R}^2_+}|f(z)|^p\omega(z)dV_\alpha(z)\right)^{q/p}.
\Ee
\item[(c)] There is a constant $C_3>0$ such that for any interval $I\subset \mathbb R$,
\Be\label{eq:main22}
|Q_I|_\alpha^{q(\frac{\gamma}{2+\alpha}-\frac{1}{p})}\left(\frac{1}{|Q_I|_\alpha}\int_{Q_I}\omega^{1-p'}(z)dV_\alpha(z)\right)^{q/p'}\mu(Q_I)\le C_3
\Ee
where $\left(\frac{1}{|Q_I|_\alpha}\int_{Q_I}\omega^{1-p'}(z)dV_\alpha(z)\right)^{1/p'}$ is understood as $\left(\inf_{Q_I}\omega\right)^{-1}$ when $p=1$.
\end{itemize}
\end{lemma}
\begin{proof}
The equivalence (a)$\Leftrightarrow$(c) is from \cite[Theorem 2.3]{Sehba3}. Clearly, (a)$\Rightarrow$(b). That (b)$\Rightarrow$(a) follows from \cite[Lemma 4.1]{Sehba3} and is the main idea in the proof of \cite[Theorem 2.3]{Sehba3}. 
\end{proof}
\vskip .2cm
We refer to \cite[Corollary 4.3]{Sehba3} for the following.
\begin{lemma}\label{fracmax}
Let $\alpha>-1$, $0\le \gamma<2+\alpha$, and let $\sigma$ be a weight. Let $1<p<\frac{2+\alpha}{\gamma}$, and define $q$ by $\frac{1}{q}=\frac{1}{p}-\frac{\gamma}{2+\alpha}$. Then there exists a constant $C=C(p,\alpha,\gamma)$ such that for any $\beta\in \{0,\frac 13\}$,
\begin{equation}
\left(\int_{\mathbb{R}_+^2}\left((\mathcal{M}_{\sigma,\alpha,\gamma}^{d,\beta}f)(z)\right)^q\sigma(z)dV_\alpha(z)\right)^{1/q}\le C\left(\int_{\mathbb{R}_+^2}|f(z)|^p\sigma(z)dV_\alpha(z)\right)^{1/p}.
\end{equation}
\end{lemma}

\subsection{Dyadic analogue of fractional Bergman operators}
Let $\alpha>-1$ and $0\le \gamma<2+\alpha$. For $\beta\in \{0,1/3\}$, we introduce the following dyadic operators \Be\label{eq:dyadicoperator}\mathcal{Q}_{\alpha,\gamma}^{\beta}f = \sum_{I\in\mathcal{D}^{\beta}}|Q_I|_{\alpha}^{\frac{\gamma}{2+\alpha}}\langle f,\frac{1_{Q_I}}{|Q_I|_{\alpha}}\rangle_{\alpha}1_{Q_I}.\Ee Here, $\langle\cdot,\cdot\rangle_\alpha$ stands for the duality pairing $$\langle f,g \rangle_{\alpha}=\int_{\mathbb{R}^2_+}f(z)\overline{g(z)}\mathrm{d}V_{\alpha}(z).$$
The operators $\mathcal{Q}_{\alpha,\gamma}^{\beta}$ were introduced in \cite{PR} in the case $\gamma=0$ in relation with sharp estimate of the Bergman projection. The following result is obtained as in the case $\gamma=0$ (see \cite[Proposition 3.4]{PR}).
\begin{lemma}\label{lem:sparsedom}
Let $\alpha>0$ and $0\le \gamma<2+\alpha$. Then exists a constant $C=C_{\alpha,\gamma}>0$ such that for all $f\in L_{loc}^1(\mathbb{R}_+^2,dV_\alpha)$, $f\ge 0$ and $z\in \mathbb{R}_+^2$,
\Be\label{eq:sparsedom}
T_{\alpha,\gamma}f(z)\le C\sum_{\beta\in\{0,\frac{1}{3}\}}\mathcal{Q}_{\alpha,\gamma}^{\beta}f(z).
\Ee
\end{lemma}
We also recall the following covering results. The first one is \cite[Lemma 3.1]{PR} while the second one is \cite[Lemma 2.3]{CarnotBenoit}.
\begin{lemma}\label{lem:dyadiccover}
Let $I$ be any interval in $\mathbb{R}$. Then the following hold.
\begin{itemize}
\item[(1)] There exists a dyadic interval $J\in \mathcal{D}^\beta$ for some $\beta\in \{0,\frac 13\}$ such that $I\subseteq J$ and $|J|\le 8|I|$.
\item[(2)] For any $\beta\in \{0,\frac 13\}$, $I$ can be covered by two adjacent intervals $I_1$ and $I_2$ in $\mathcal{D}^\beta$ such that $|I|<|I_1|=|I_2|\leq 2|I|$.
\end{itemize}
\end{lemma}
\vskip .2cm
\begin{remark}\label{rem:remarq1}
As the operators considered here are positive, we only need to consider positive functions in our proofs. Also from Lemma \ref{lem:sparsedom}, it follows that to prove the norm inequalities for $T_{\alpha,\gamma}$, it suffices to prove them for the positive dyadic operators $\mathcal{Q}_{\alpha,\gamma}^{\beta}$. Let us note that it is also enough to prove the norm inequalities for bounded and compactly supported functions as the general case will follow from Fatou's lemma. 
%Moreover, from Lemma \ref{lem:dyadiccover} and (\ref{eq:dyadicoperator}), we can further suppose that our functions are supported on dyadic Carleson squares.
\end{remark}
\section{Proof of Theorem \ref{thm:main3} and Theorem \ref{thm:main4}}

\begin{proof}[Proof of Theorem \ref{thm:main3}]
Assume that $\omega\in B_{1,q,\alpha}.$ Following Remark \ref{rem:remarq1} one only needs to show that the estimate \eqref{eq1} holds with $T_{\alpha,\gamma}$ replaced by $\mathcal{Q}_{\alpha,\gamma}^{\beta}.$ From the same remark, the definition (\ref{eq:dyadicoperator}) and Lemma \ref{lem:dyadiccover}, we can assume that $f$ is supported on some dyadic square $Q_J,\,\, J\in\mathcal{D}^{\beta}$, as we also prove that the estimate obtained is independent of $J$. 

Put for $\lambda>0,$ $$\mathbb{E}_{\lambda}:=\{z\in\mathbb{R}_+^2:\mathcal{Q}_{\alpha,\gamma}^{\beta}f(z)>\lambda\}.$$ Let $\Lambda > 0$ be fixed. We first check that $$\sup_{0<\lambda<\Lambda}\lambda^q|\mathbb{E}_{\lambda}|_{u,\alpha}<\infty,\quad u=\omega^q.$$ Indeed if $z\notin Q_J,$ then $\mathcal{Q}_{\alpha,\gamma}^{\beta}f(z)\ne 0$ only if we can find $I\in\mathcal{D}^{\beta}$ such that $J\subset I$ and $z\in Q_I.$ Let $I_0$ be the smallest $I$ such that this holds. Then for any other $I\in\mathcal{D}^\beta$ such that $J\subset I$ and $z\in Q_I,$ we have that $Q_{I_0}\subseteq Q_I$ and so, $|Q_I|_\alpha= 2^{(2+\alpha)k}|Q_{I_0}|_{\alpha}$ for some integer $k>0.$ Moreover,  $$\int_{Q_I}fdV_{\alpha} = \int_{Q_{I_0}}fdV_{\alpha}.$$ It follows that 
\begin{eqnarray*}
\mathcal{Q}_{\alpha,\gamma}^{\beta}f(z) & = & \sum_{J\subset I}|Q_I|^{\frac{\gamma}{2+\alpha} - 1}_{\alpha}\int_{Q_I}fdV_{\alpha}\\
				& = & \left(\sum_{k=0}^{\infty}2^{(1+\alpha)k\left(\frac{\gamma}{2+\alpha} -1\right)}\right)|Q_{I_0}|^{\frac{\gamma}{2+\alpha}-1}_{\alpha}\int_{Q_{I_0}}fdV_{\alpha}\\
				& \le & C\mathcal{M}^{d,\beta}_{\alpha,\gamma}f(z).
\end{eqnarray*}
%where $\mathcal{M}^{d,\beta}_{\alpha,\gamma}$ is the dyadic analogue of $\mathcal{M}_{\alpha,\gamma}$ which is defined as $\mathcal{M}_{\alpha,\gamma}$ but with the supremum taken only over dyadic intervals in $\mathcal{D}^\beta$.
Hence putting $$\mathbb{F}_{\lambda}:=\left\{ z\in\mathbb{R}^2_+ : \mathcal{M}^{d,\beta}_{\alpha,\gamma}f(z) > \frac{\lambda}{C}\right\},$$ we obtain 
\begin{eqnarray*}
\sup_{0<\lambda<\Lambda}\lambda^q|\mathbb{E}_{\lambda}|_{u,\alpha} & \le & \Lambda^q |Q_J|_{u,\alpha}+ \sup_{0<\lambda <\Lambda}\lambda^q\left|\left\{z\in\mathbb{R}_+^2\setminus Q_J : \mathcal{M}^{d,\beta}_{\alpha,\gamma}f(z)>\frac{\lambda}{C}\right\}\right|_{u,\alpha}\\
& \le & \Lambda^q |Q_J|_{u,\alpha}+ \sup_{0<\lambda <\Lambda}\lambda^q|\mathbb{F}_{\lambda}|_{u,\alpha}< \infty
\end{eqnarray*}
since $u$ is locally integrable and $\mathcal{M}_{\alpha,\gamma}^{d,\beta}$ is bounded from $L^1(u\mathrm{d}V_{\alpha})$ to $L^{q,\infty}(u\mathrm{d}V_{\alpha})$ for $u$ satisfying $$\sup_I\mbox{ess}\sup_{I\in Q_I}\left(\frac{1}{|Q_I|_{\alpha}}\int_{Q_I}u\mathrm{d}V_{\alpha}\right)^{\frac{1}{q}}\omega^{-1}(z) <\infty$$ (see Lemma \ref{lem:weakfrac}). 

Next, we observe that $$\mathcal{Q}_{\alpha,\gamma}^{\beta}f(z) = \mathcal{Q}_{\alpha,\gamma,J}^{\beta,\,\,\mbox{in}}f(z) + \mathcal{Q}_{\alpha,\gamma,J}^{\beta,\,\,\mbox{out}}f(z)$$ where $$\mathcal{Q}_{\alpha,\gamma,J}^{\beta,\,\,\mbox{in}}f(z) = \sum_{\substack{I\in\mathcal{D}^{\beta}\\I\subseteq J}}|Q_I|_{\alpha}^{\frac{\gamma}{2+\alpha}}\left(\frac{1}{|Q_I|_{\alpha}}\int_{Q_I}f\mathrm{d}V_{\alpha}\right)1_{Q_I}(z)$$ and $$\mathcal{Q}_{\alpha,\gamma,J}^{\beta,\,\,\mbox{out}}f(z) = \sum_{\substack{I\in\mathcal{D}^{\beta}\\I\supseteq J}}|Q_I|_{\alpha}^{\frac{\gamma}{2+\alpha}}\left(\frac{1}{|Q_I|_{\alpha}}\int_{Q_I}f\mathrm{d}V_{\alpha}\right)1_{Q_I}(z).$$ We recall that $\mathbb{E}_{\lambda}$ can be written as a union of maximal dyadic Carleson squares. Indeed, if $z\in\mathbb{E}_{\lambda},$ denote by $Q(z)$ the smallest dyadic square containing $z.$ Let $w \in Q(z).$ Then any dyadic square supported by an interval in $\mathcal{D}^\beta$ containing $z$ contains $Q(z)$ and hence contains $w$. Thus 
\begin{eqnarray*}
\lambda < \mathcal{Q}_{\alpha,\gamma}^{\beta}f(z) & = & \sum_{\substack{I\in\mathcal{D}^{\beta}\\z\in Q_I}}|Q_I|_{\alpha}^{\frac{\gamma}{2+\alpha}}\langle f,\frac{1_{Q_I}}{|Q_I|_{\alpha}}\rangle_{\alpha}\\
& \le & \sum_{\substack{I\in\mathcal{D}^{\beta}\\w\in Q_I}}|Q_I|_{\alpha}^{\frac{\gamma}{2+\alpha}}\langle f,\frac{1_{Q_I}}{|Q_I|_{\alpha}}\rangle_{\alpha}\\
& = & \mathcal{Q}_{\alpha,\gamma}^{\beta}f(w).
\end{eqnarray*}
Hence $Q(z)\subset\mathbb{E}_{\lambda}.$ That is for any $z\in\mathbb{E}_{\lambda},$ there is a dyadic square containing $z$ that is entirely contained in $\mathbb{E}_{\lambda}$. Moreover, as $f$ is compactly supported, this dyadic square cannot be arbitrary large. Thus $\mathbb{E}_{\lambda}$ is a union of maximal dyadic Carleson squares. 
\vskip .2cm
Let $I\in\mathcal{D}^\beta$ be such that $Q_I$ is one of the maximal squares above. Let $\tilde{I}$ be the dyadic parent of $I.$ Then there exists $z_0\in Q_{\tilde{I}}\backslash Q_I$ such that for any $z\in Q_I,$ $$\lambda \ge \mathcal{Q}_{\alpha,\gamma}^{\beta}f(z_0) \ge \mathcal{Q}_{\alpha,\gamma,I}^{\beta,\,\,\mbox{out}}f(z_0) = \mathcal{Q}_{\alpha,\gamma,I}^{\beta,\,\,\mbox{out}}f(z).$$ It follows that for any $z\in Q_J\cap\mathbb{E}_{2\lambda},$ $$\mathcal{Q}_{\alpha,\gamma,I}^{\beta,\,\,\mbox{in}}f(z) > \lambda.$$ Next we fix $\Lambda >0.$ We also fix $\lambda$ such that $0<\lambda<\Lambda$. We recall that $u=\omega^q.$ Define $\mathcal{L}$ to be family of maximal dyadic Carleson squares whose union is $\mathbb{E}_{\lambda}.$ Put 
\begin{equation*}
\mathcal{L}_1 : =  \{ Q_I\in \mathcal{L}: |Q_I\cap\mathbb{E}_{2\lambda}|_{u,\alpha}\ge2^{-q-1}|Q_I|_{u,\alpha}\}
\end{equation*}
and 
\begin{equation*}
\mathcal{L}_2 :=  \mathcal{L}\backslash\mathcal{L}_1.
\end{equation*}
Then
\begin{eqnarray*}
(2\lambda)^q|\mathbb{E}_{2\lambda}|_{u,\alpha} & = & (2\lambda)^q\sum_{Q_I\in\mathcal{L}}|Q_I\cap\mathbb{E}_{2\lambda}|_{u,\alpha}\\
& = & (2\lambda)^q\left(\sum_{Q_I\in\mathcal{L}_1}|Q_I\cap\mathbb{E}_{2\lambda}|_{u,\alpha} + \sum_{Q_I\in\mathcal{L}_2}|Q_I\cap\mathbb{E}_{2\lambda}|_{u,\alpha}\right).
\end{eqnarray*}
We have
\begin{eqnarray*}
(2\lambda)^q\sum_{Q_I\in\mathcal{L}_2}|Q_I\cap\mathbb{E}_{2\lambda}|_{u,\alpha} & \le & (2\lambda)^q2^{-q-1}\sum_{Q_I\in\mathcal{L}_2}|Q_I|_{u,\alpha}\\
& \le & \frac{1}{2}\lambda^q\sum_{Q_I\in\mathcal{L}_2}|Q_I|_{u,\alpha}\\
& \le & \frac{1}{2}\lambda^q|\mathbb{E}_{\lambda}|_{u,\alpha}\\
& \le & \frac{1}{2}\sup_{0<\lambda<\Lambda}\lambda^q|\mathbb{E}_{\lambda}|_{u,\alpha} <\infty.
\end{eqnarray*}
Now
\begin{eqnarray*}
L & : = & (2\lambda)^q\sum_{Q_I\in\mathcal{L}_1}|Q_I\cap\mathbb{E}_{2\lambda}|_{u,\alpha}\\
& = & (2\lambda)^q\sum_{Q_I\in\mathcal{L}_1}|Q_I\cap\mathbb{E}_{2\lambda}|_{u,\alpha}^q|Q_I\cap\mathbb{E}_{2\lambda}|_{u,\alpha}^{1-q}\\
& \le & 2^{q^2-1}(2\lambda)^q\sum_{Q_I\in\mathcal{L}_1}|Q_I\cap\mathbb{E}_{2\lambda}|_{u,\alpha}^q|Q_I|_{u,\alpha}^{1-q}\\
& \le & C(q)\sum_{Q_I\in\mathcal{L}_1}(\lambda|Q_I\cap\mathbb{E}_{2\lambda}|_{u,\alpha})^q|Q_I|_{u,\alpha}^{1-q}\\
& \le & C(q)\sum_{Q_I\in\mathcal{L}_1}(\lambda|\{z\in Q_I:\mathcal{Q}_{\alpha,\gamma,I}^{\beta,\,\,\mbox{in}}f(z)>\lambda\}|_{u,\alpha})^q|Q_I|_{u,\alpha}^{1-q}\\
& \le & C(q)\sum_{Q_I\in\mathcal{L}_1}\left(\int_{Q_I}\left(\mathcal{Q}_{\alpha,\gamma,I}^{\beta,\,\,\mbox{in}}f(z)\right)u(z)\mathrm{d}V_{\alpha}(z)\right)^q|Q_I|_{u,\alpha}^{1-q}\\
& = & C(q)\sum_{Q_I\in\mathcal{L}_1}\left(\int_{Q_I}\left(\mathcal{Q}_{\alpha,\gamma,I}^{\beta,\,\,\mbox{in}}u(z)\right)f(z)\mathrm{d}V_{\alpha}(z)\right)^q|Q_I|_{u,\alpha}^{1-q}
\end{eqnarray*}
where we have used duality and the fact that $\mathcal{Q}_{\alpha,\gamma,I}^{\beta,\,\,\mbox{in}}$ is self-adjoint with respect to the pairing $$\langle f,g \rangle_{\alpha}=\int_{\mathbb{R}^2_+}f(z)\overline{g(z)}\mathrm{d}V_{\alpha}(z).$$ 
Let us estimate $\mathcal{Q}_{\alpha,\gamma,I}^{\beta,\,\,\mbox{in}}u.$ We recall that $$q = \frac{2+\alpha}{2+\alpha-\gamma}\qquad\mbox{and so}\qquad\frac{1}{q'} = \frac{\gamma}{2+\alpha}.$$ We first write 
\begin{eqnarray*}
\mathcal{Q}_{\alpha,\gamma,I}^{\beta,\,\,\mbox{in}}u(z) & = & \sum_{\substack{K\in\mathcal{D}^{\beta}\\ K\subseteq I}}|Q_K|^{\frac{\gamma}{2+\alpha}}_{\alpha}\left(\frac{1}{|Q_K|_{\alpha}}\int_{Q_K}u\mathrm{d}V_{\alpha}\right)1_{Q_K}(z)\\
& = & \sum_{\substack{K\in\mathcal{D}^{\beta}\\ K\subseteq I}}|Q_K|^{\frac{1}{q'}}_{\alpha}\left(\frac{1}{|Q_K|_{\alpha}}\int_{Q_K}u\mathrm{d}V_{\alpha}\right)^{\frac{1}{q}+\frac{1}{q'}}1_{Q_K}(z).
\end{eqnarray*}
We observe that for any $Q_K\subseteq Q_I,\,\,z\in Q_K,$ $$\left(\frac{1}{|Q_K|_{\alpha}}\int_{Q_K}u\mathrm{d}V_{\alpha}\right)^{\frac{1}{q}}\le [\omega]_{B_{1,q,\alpha}}\omega(z).$$ On the other hand, as $u\in B_{1,\alpha}\subset B_{2,\alpha}$ with $[u]_{B_{2,\alpha}}\le [u]_{B_{1,\alpha}}$, we have from Lemma \ref{carnot} and Lemma \ref{l4} that for any $K\in\mathcal{D}^{\beta},$ $$\frac{|B_K|_{u,\alpha}}{|Q_K|_{u,\alpha}}\le\delta$$ where $\delta = 1- \frac{1}{C_\alpha[u]_{B_{1,\alpha}}}>\frac{1}{2}$ where $C_\alpha:=\max\{2,\frac{2^{2+2\alpha}}{(2^{1+\alpha}-1)^2}\}$, $B_K$ being the lower half of $Q_K.$ It follows that if $K_j$ is a descendant of $K$ of the $j-$th generation, then $$\frac{|Q_{K_j}|_{u,\alpha}}{|Q_K|_{u,\alpha}}\le\delta^j.$$ Let us fix $z\in Q_I.$ Then each $K$ in the sum $\mathcal{Q}_{\alpha,\gamma,I}^{\beta,\,\,\mbox{in}}u(z)$ is a descendant of $I$ of some generation and it is the unique dyadic interval of this  generation such that $z\in Q_K.$ It follows that 
\begin{eqnarray*}
\sum_{\substack{K\in\mathcal{D}^{\beta}\\K\subseteq I}}|Q_K|^{\frac{1}{q'}}_{\alpha}\left(\frac{1}{|Q_K|_{\alpha}}\int_{Q_K}u\mathrm{d}V_{\alpha}\right)^{\frac{1}{q'}} & = & \sum_{K\subseteq I}|Q_K|^{\frac{1}{q'}}_{u,\alpha}\\
& \le & \sum_{k=0}^{\infty}\delta^{\frac{k}{q'}}|Q_I|^{\frac{1}{q'}}_{u,\alpha}\\
& \le & \frac{1}{1-\delta^{\frac{1}{q'}}}|Q_I|^{\frac{1}{q'}}_{u,\alpha}.
\end{eqnarray*}
As $\delta >\frac{1}{2},$ $$\frac{1}{1-\delta^{\frac{1}{q'}}}\le \left(\frac{1}{1-\delta}\right)^{\frac{1}{q'}} = C_\alpha^{\frac{1}{q'}}[\omega]^{\frac{q}{q'}}_{B_{1,q,\alpha}}.$$ Hence \begin{equation}\label{eq:estimQbeta}\mathcal{Q}_{\alpha,\gamma,I}^{\beta,\,\,\mbox{in}}u(z)\le C(q,\alpha)[\omega]^{1+\frac{q}{q'}}_{B_{1,q,\alpha}}\omega(z)|Q_I|^{\frac{1}{q'}}_{u,\alpha}.
\end{equation} 
It follows that 
\begin{eqnarray*}
L & : = & (2\lambda)^q\sum_{Q_I\in\mathcal{L}_1}|Q_I\cap\mathbb{E}_{2\lambda}|_{u,\alpha}\\
& \le & C(q,\alpha)[\omega]^{q+\frac{q^2}{q'}}_{B_{1,q,\alpha}}\sum_{Q_I\in\mathcal{L}_1}\left(\int_{Q_I}f(z)\omega(z)\mathrm{d}V_{\alpha}(z)\right)^q|Q_I|^{1-q+\frac{q}{q'}}_{u,\alpha}\\
& = & C(q,\alpha)[\omega]^{q^2}_{B_{1,q,\alpha}}\left(\sum_{Q_I\in\mathcal{L}_1}\int_{Q_I}f\omega\mathrm{d}V_{\alpha}(z)\right)^q\\
& \le & C(\alpha)[\omega]^{q^2}_{B_{1,q,\alpha}}\left(\int_{\mathbb{R}_+^2}f\omega\mathrm{d}V_{\alpha}(z)\right)^q.
\end{eqnarray*}
Putting the two estimates together, we obtain \Be\label{eq:supsup}(2\lambda)^q|\mathbb{E}_{2\lambda}|_{u,\alpha}\le \frac{1}{2}\sup_{0<\lambda<\Lambda}\lambda^q|\mathbb{E}_{\lambda}|_{u,\alpha} + C[\omega]^{q^2}_{B_{1,q,\alpha}}\left(\int_{\mathbb{R}_+^2}f\omega\mathrm{d}V_{\alpha}(z)\right)^q.\Ee Recall that $\sup_{0<\lambda<\Lambda}\lambda^q|\mathbb{E}_{\lambda}|_{u,\alpha}<\infty$. Hence, taking the supremum on the left hand side of the inequality (\ref{eq:supsup}), we get $$\sup_{0<\lambda<\Lambda}\lambda^q|\mathbb{E}_{\lambda}|_{u,\alpha} \le C[\omega]^{q^2}_{B_{1,q,\alpha}}\left(\int_{\mathbb{R}_+^2}f\omega\mathrm{d}V_{\alpha}(z)\right)^q.$$ Letting $\Lambda\rightarrow\infty$, we obtain the estimate \eqref{eq1}.
\end{proof}
We next provide the modifications needed in the above proof to prove Theorem \ref{thm:main4}.
\begin{proof}[Proof of Theorem \ref{thm:main4}]
Assume that $\omega\in B_{1,q,\alpha}.$ We recall with Lemma \ref{lem:sparsedom} that for $f\ge 0,$ $$P_{\alpha}^+f\le \sum_{\beta\in\{0,\frac{1}{3}\}}\mathcal{Q}_{\alpha}^{\beta}f,$$ where $$\mathcal{Q}_{\alpha}^{\beta}f = \sum_{I\in\mathcal{D}^{\beta}}\langle f,\frac{1_{Q_I}}{|Q_I|_{\alpha}}\rangle_{\alpha}\chi_{Q_I}.$$ We also recall with Remark \ref{rem:remarq1} that one only needs to show that the estimate \eqref{eq41} holds with $P_{\alpha}^+$ replaced by $\mathcal{Q}_{\alpha}^{\beta}.$ We still assume that $f$ is supported on some dyadic cube $Q_J,\,\, J\in\mathcal{D}^{\beta}$. We recall that $\eta=(2+\alpha)(q-1)+\alpha$. Put for $\lambda>0,$ $$\mathbb{E}_{\lambda}:=\{z\in\mathbb{R}_+^2:\mathcal{Q}_{\alpha}^{\beta}f(z)>\lambda\}.$$ Let $\Lambda > 0$ be fixed. Let us check as above that $$\sup_{0<\lambda<\Lambda}\lambda^q|\mathbb{E}_{\lambda}|_{u,\eta}<\infty,\quad u=\omega^q.$$ 
 
Still reasoning as in the previous proof, we obtain that there is a constant $C>0$ such that for $z\notin Q_J$, $\mathcal{Q}_{\alpha}^{\beta}f(z)\le C\mathcal{M}_\alpha^{d,\beta} f(z)$, and putting
$$\mathbb{F}_{\lambda}:=\left\{z\in\mathbb{R}_+^2|Q_J : \mathcal{M}^{d,\beta}_{\alpha}f(z)>\frac{\lambda}{C}\right\},$$ we obtain
\begin{eqnarray*}
\sup_{0<\lambda<\Lambda}\lambda^q|\mathbb{E}_{\lambda}|_{u,\eta} & \le & \Lambda^q |Q_J|_{u,\eta}+ \sup_{0<\lambda <\Lambda}\lambda^q|\mathbb{F}_{\lambda}|_{u,\eta}< \infty
\end{eqnarray*}
since $u$ is locally integrable and by Lemma \ref{lem:weakfrac}, $\mathcal{M}_{\alpha}^{d,\beta}$ is bounded from $L^1(u\mathrm{d}V_{\alpha})$ to $L^{q,\infty}(udV_{\eta})$ as the measure $d\mu(z)=u(z)dV_\eta(z)$ is such that for any interval $I$ and any $z\in Q_I$,
\Beas
|Q_I|_\alpha^{-q}\omega^{-q}(z)\mu(Q_I) &=& |Q_I|_\alpha^{-q}\omega^{-q}(z)|Q_I|_{u,\eta}\\ &\le&  |Q_I|_\alpha^{-q}\omega^{-q}(z)|Q_I|_{\alpha}^{q-1}|Q_I|_{u,\alpha}\\ &=& \frac{|Q_I|_{u,\alpha}}{|Q_I|_{\alpha}}\omega^{-q}(z)\\ &\le& [\omega]_{1,q,\alpha}^q.
\Eeas
  Now let us decompose $\mathcal{Q}_\alpha^{\beta}$ as follows $$\mathcal{Q}_{\alpha}^{\beta}f(z) = \mathcal{Q}_{\alpha,J}^{\beta,\,\,\mbox{in}}f(z) + \mathcal{Q}_{\alpha,J}^{\beta,\,\,\mbox{out}}f(z)$$ where $$\mathcal{Q}_{\alpha,J}^{\beta,\,\,\mbox{in}}f(z) = \sum_{\substack{I\in\mathcal{D}^{\beta}\\I\subseteq J}}\left(\frac{1}{|Q_I|_{\alpha}}\int_{Q_I}f\mathrm{d}V_{\alpha}\right)1_{Q_I}(z)$$ and $$\mathcal{Q}_{\alpha,J}^{\beta,\,\,\mbox{out}}f(z) = \sum_{\substack{I\in\mathcal{D}^{\beta}\\I\supseteq J}}\left(\frac{1}{|Q_I|_{\alpha}}\int_{Q_I}f\mathrm{d}V_{\alpha}\right)1_{Q_I}(z).$$ 
\vskip .3cm 
  We also obtain as in the previous proof that for any $z\in Q_J\cap\mathbb{E}_{2\lambda},$ $$\mathcal{Q}_{\alpha,J}^{\beta,\,\,\mbox{in}}f(z) > \lambda.$$ Let us once more fix $\Lambda >0.$ We then also fix $\lambda$ such that $0<\lambda<\Lambda$. We still denote by $\mathcal{L}$ the  family of maximal dyadic Carleson squares whose union is $\mathbb{E}_{\lambda}.$ We also define
\begin{equation*}
\mathcal{L}_1 : =  \{ Q_I\in \mathcal{L}: |Q_I\cap\mathbb{E}_{2\lambda}|_{u,\eta}\ge2^{-q-1}|Q_I|_{u,\eta}\}
\end{equation*}
and 
\begin{equation*}
\mathcal{L}_2 :=  \mathcal{L}\backslash\mathcal{L}_1.
\end{equation*}
Then
\begin{eqnarray*}
(2\lambda)^q|\mathbb{E}_{2\lambda}|_{u,\eta} 
& = & (2\lambda)^q\left(\sum_{Q_I\in\mathcal{L}_1}|Q_I\cap\mathbb{E}_{2\lambda}|_{u,\eta} + \sum_{Q_I\in\mathcal{L}_2}|Q_I\cap\mathbb{E}_{2\lambda}|_{u,\eta}\right).
\end{eqnarray*}
We obtain once more that
\begin{eqnarray*}
(2\lambda)^q\sum_{Q_I\in\mathcal{L}_2}|Q_I\cap\mathbb{E}_{2\lambda}|_{u,\eta} 
& \le & \frac{1}{2}\sup_{0<\lambda<\Lambda}\lambda^q|\mathbb{E}_{\lambda}|_{u,\eta} <\infty.
\end{eqnarray*}
Now, still following the proof of Theorem \ref{thm:main3}, we obtain
\begin{eqnarray*}
L & : = & (2\lambda)^q\sum_{Q_I\in\mathcal{L}_1}|Q_I\cap\mathbb{E}_{2\lambda}|_{u,\eta}\\
& \le & C(q)\sum_{Q_I\in\mathcal{L}_1}\left(\int_{Q_I}\left(\mathcal{Q}_{\alpha,I}^{\beta,\,\,\mbox{in}}f(z)\right)u(z)\mathrm{d}V_{\eta}(z)\right)^q|Q_I|_{u,\eta}^{1-q}\\ &=& C(q)\sum_{Q_I\in\mathcal{L}_1}\left(\sum_{K\subseteq I}\langle f,\frac{1_{Q_K}}{|Q_K|_\alpha}\rangle_\alpha |Q_K|_{u,\eta}\right)^q|Q_I|_{u,\eta}^{1-q}.
\end{eqnarray*}
Now observe that as $1-q<0$ and $T_I\subset Q_I$, we have 
$$|Q_I|_{u,\eta}^{1-q}\le |T_I|_{u,\eta}^{1-q}\le C_{\alpha,\gamma}|I|^{(2+\alpha)(1-q)(q-1)}|T_I|_{u,\alpha}^{1-q}.$$
Hence for any $K\subset I$, we obtain
$$|Q_I|_{u,\eta}^{1-q}\le C_{\alpha,\gamma}|Q_K|_{\alpha}^{(1-q)(q-1)}|T_I|_{u,\alpha}^{1-q}.$$
Note also that $$|Q_K|_{u,\eta}\le |Q_K|_{\alpha}^{q-1}|Q_K|_{u,\alpha}.$$
It follows from these observations that
\begin{eqnarray*}
L &\le&   C(q)\sum_{Q_I\in\mathcal{L}_1}\left(\sum_{K\subseteq I}|Q_K|_{\alpha}^{\frac{1}{q'}}\langle f,\frac{1_{Q_K}}{|Q_K|_\alpha}\rangle_\alpha |Q_K|_{u,\alpha}\right)^q|T_I|_{u,\alpha}^{1-q}\\ &=&  C(q)\sum_{Q_I\in\mathcal{L}_1}\left(\int_{Q_I}\left(\mathcal{Q}_{\alpha,\gamma,I}^{\beta,\,\,\mbox{in}}f(z)\right)u(z)\mathrm{d}V_{\alpha}(z)\right)^q|T_I|_{u,\alpha}^{1-q}.
%& = & c(q)\sum_{Q_I\in\mathcal{L}_1}\left(\int_{Q_I}\left(\mathcal{Q}_{\alpha,\gamma,I}^{\beta,\,\,\mbox{in}}u(z)\right)f(z)\mathrm{d}V_{\alpha}(z)\right)^q|T_I|_{u,\alpha}^{1-q}.
\end{eqnarray*}
Hence
\begin{equation}\label{eq:L}
L\le C(q)\sum_{Q_I\in\mathcal{L}_1}\left(\int_{Q_I}\left(\mathcal{Q}_{\alpha,\gamma,I}^{\beta,\,\,\mbox{in}}u(z)\right)f(z)\mathrm{d}V_{\alpha}(z)\right)^q|T_I|_{u,\alpha}^{1-q}.
\end{equation}
Using (\ref{eq:estimQbeta}) and Lemma \ref{carnot}, we obtain that
\Beas \mathcal{Q}_{\alpha,\gamma,I}^{\beta,\,\,\mbox{in}}u(z) &\le&  C(q,\alpha)[\omega]^{1+\frac{q}{q'}}_{B_{1,q,\alpha}}\omega(z)|Q_I|^{\frac{1}{q'}}_{u,\alpha}\\ &\le& C(q,\alpha)[\omega]^{1+\frac{2q}{q'}}_{B_{1,q,\alpha}}\omega(z)|T_I|^{\frac{1}{q'}}_{u,\alpha}.
\Eeas
Taking this into (\ref{eq:L}), we conclude that
\Beas
L & : = & (2\lambda)^q\sum_{Q_I\in\mathcal{L}_1}|Q_I\cap\mathbb{E}_{2\lambda}|_{u,\eta}\\ &\le& C(q,\alpha)[\omega]^{q+\frac{2q^2}{q'}}_{B_{1,q,\alpha}}\sum_{Q_I\in\mathcal{L}_1}\left(\int_{Q_I}f(z)\omega(z)\mathrm{d}V_{\alpha}(z)\right)^q|T_I|_{u,\alpha}^{\frac{q}{q'}+1-q}\\ &\le& C(q,\alpha)[\omega]^{2q^2-q}_{B_{1,q,\alpha}}\left(\int_{\mathbb{R}_+^2}f(z)\omega(z)\mathrm{d}V_{\alpha}(z)\right)^q.
\Eeas
The remaining of the proof then follows as in the last part of the proof of Theorem \ref{thm:main3}.
\end{proof}
\section{Proof of Theorem \ref{thm:main1} and Theorem \ref{thm:main2}}
We start this section with the proof of Theorem \ref{thm:main1}. 

\begin{proof}[Proof of Theorem \ref{thm:main1}]
We start by considering the sufficiency. We recall that for any $f\ge 0,$ $$T_{\alpha,\gamma}f\le \sum_{\beta\in\{0,\frac{1}{3}\}}\mathcal{Q}_{\alpha,\gamma}^{\beta}f$$ where the dyadic operators $\mathcal{Q}_{\alpha,\gamma}^{\beta}$ are given by (\ref{eq:dyadicoperator}). Thus the question reduces to proving that $$\mathcal{Q}_{\alpha,\gamma}^{\beta} : L^p(\omega^p\mathrm{d}V_{\alpha})\longrightarrow L^q(\omega^q\mathrm{d}V_{\alpha})$$ is bounded. Let us put $u=\omega^q$ and $\sigma=\omega^{-p'}.$ For any $g\in L^{q'}(\omega^{-q'}\mathrm{d}V_{\alpha}),$ we would like to estimate \begin{eqnarray*}
\langle T_{\alpha,\gamma}f,g\rangle_{\alpha} & = & \sum_{I\in\mathcal{D}^{\beta}}|Q_I|^{\frac{\gamma}{2+\alpha}-1}_{\alpha}\left(\int_{Q_I}f\mathrm{d}V_{\alpha}\right)\left(\int_{Q_I}g\mathrm{d}V_{\alpha}\right)\\
& = & \sum_{I\in\mathcal{D}^{\beta}}|Q_I|^{\frac{\gamma}{2+\alpha}-1}_{\alpha}\left(\int_{Q_I}(f\sigma^{-1})\sigma\mathrm{d}V_{\alpha}\right)\left(\int_{Q_I}(gu^{-1})u\mathrm{d}V_{\alpha}\right).
\end{eqnarray*}
Put $$\mathcal{S}_{\sigma,\alpha}(f,Q_I) = \frac{1}{|Q_I|_{\sigma,\alpha}}\int_{Q_I}f\sigma\mathrm{d}V_{\alpha}$$ 
and $$\mathcal{S}_{u,\alpha,\gamma}(g,Q_I) = \frac{1}{|Q_I|_{u,\alpha}^{1-\frac{\gamma}{2+\alpha}}}\int_{Q_I}gu\mathrm{d}V_{\alpha}.$$ 
Then 
\begin{eqnarray*}
L & : = & \langle T_{\alpha,\gamma}f,g\rangle_{\alpha}\\ &=& \sum_{I\in\mathcal{D}^{\beta}}|Q_I|^{\frac{\gamma}{2+\alpha}-1}_{\alpha}|Q_I|_{\sigma,\alpha}|Q_I|_{u,\alpha}^{1-\frac{\gamma}{2+\alpha}}\mathcal{S}_{\sigma,\alpha}(f\sigma^{-1},Q_I)\mathcal{S}_{u,\alpha}(gu^{-1},Q_I).
\end{eqnarray*}
Now observe that if $r=1+\frac{q}{p'}$, then $[u]_{\mathcal{B}_{r,\alpha}}=[\omega]_{B_{p,q,\alpha}}^q$ and $[\sigma]_{\mathcal{B}_{r',\alpha}}=[\omega]_{B_{p,q,\alpha}}^{p'}$. It follows using Lemma \ref{carnot} that
\Beas
|Q_I|^{\frac{\gamma}{2+\alpha}-1}_{\alpha}|Q_I|_{\sigma,\alpha}|Q_I|_{u,\alpha}^{1-\frac{\gamma}{2+\alpha}} &=& |Q_I|^{-\frac{1}{p'}-\frac{1}{q}}_{\alpha}|Q_I|_{\sigma,\alpha}^{\frac{1}{p'}+\frac{1}{p}}|Q_I|_{u,\alpha}^{\frac{1}{p'}+\frac{1}{q}}\\ &\le& [\omega]_{B_{p,q,\alpha}}|Q_I|_{\sigma,\alpha}^{\frac{1}{p}}|Q_I|_{u,\alpha}^{\frac{1}{p'}}\\ &\le& C_{p,\alpha,\gamma}[\omega]_{B_{p,q,\alpha}}[u]_{\mathcal{B}_{r,\alpha}}^{\frac{1}{p'}}[\sigma]_{\mathcal{B}_{r',\alpha}}^{\frac{1}{p}}|T_I|_{\sigma,\alpha}^{\frac{1}{p}}|T_I|_{u,\alpha}^{\frac{1}{p'}}\\ &=& C_{p,\alpha,\gamma}[\omega]^{1+\frac{p'}{p} + \frac{q}{p'}}_{B_{p,q,\alpha}}|T_I|_{\sigma,\alpha}^{\frac{1}{p}}|T_I|_{u,\alpha}^{\frac{1}{p'}}.
\Eeas
Hence
\begin{eqnarray*}
L  &:=& \langle T_{\alpha,\gamma}f,g\rangle_{\alpha}\\ & \le & C[\omega]^{1+\frac{p'}{p} + \frac{q}{p'}}_{B_{p,q,\alpha}}\sum_{I\in\mathcal{D}^{\beta}}\mathcal{S}_{\sigma,\alpha}(f\sigma^{-1},Q_I)|T_I|^{\frac{1}{p}}_{\sigma,\alpha}\mathcal{S}_{u,\alpha}(gu^{-1},Q_I)|T_I|^{\frac{1}{p'}}_{u,\alpha}\\ &\le& C[\omega]^{1+\frac{p'}{p} + \frac{q}{p'}}_{B_{p,q,\alpha}}L_1\times L_2
\end{eqnarray*}
where
$$L_1:=\left(\sum_{I\in\mathcal{D}^{\beta}}\mathcal{S}_{\sigma,\alpha}(f\sigma^{-1},Q_I)^p|T_I|_{\sigma,\alpha} \right)^{\frac{1}{p}}$$
and 
$$L_2:=\left(\sum_{I\in\mathcal{D}^{\beta}}\mathcal{S}_{u,\alpha}(fu^{-1},Q_I)^{p'}|T_I|_{u,\alpha} \right)^{\frac{1}{p'}}.$$
We easily obtain with the help of Lemma \ref{fracmax} that
\Beas
L_1 &:=& \left(\sum_{I\in\mathcal{D}^{\beta}}\mathcal{S}_{\sigma,\alpha}(f\sigma^{-1},Q_I)^p|T_I|_{\sigma,\alpha} \right)^{\frac{1}{p}}\\ &=& \left(\sum_{I\in\mathcal{D}^{\beta}}\int_{T_I}\mathcal{S}_{\sigma,\alpha}(f\sigma^{-1},Q_I)^p\sigma\mathrm{d}V_{\alpha} \right)^{\frac{1}{p}}\\ &\le& \left(\int_{\mathbb{R}_+^2}(\mathcal{M}_{\sigma,\alpha}^{d,\beta}(f\sigma^{-1})(z))^p\sigma(z)\mathrm{d}V_{\alpha}(z) \right)^{\frac{1}{p}}\\ &\le& C\left(\int_{\mathbb{R}_+^2}(f\sigma^{-1})^p\sigma\mathrm{d}V_{\alpha} \right)^{\frac{1}{p}} = C\left(\int_{\mathbb{R}_+^2}(f\omega)^p\mathrm{d}V_{\alpha} \right)^{\frac{1}{p}}.
\Eeas
Observing that $\frac{1}{q'}-\frac{1}{p'}=\frac{\gamma}{2+\alpha}$, we obtain with the help of Lemma \ref{fracmax} that
\Beas
L_2 &:=& \left(\sum_{I\in\mathcal{D}^{\beta}}\mathcal{S}_{u,\alpha}(fu^{-1},Q_I)^{p'}|T_I|_{u,\alpha} \right)^{\frac{1}{p'}}\\ &=& \left(\sum_{I\in\mathcal{D}^{\beta}}\int_{T_I}\mathcal{S}_{u,\alpha}(gu^{-1},Q_I)^{p'}u\mathrm{d}V_{\alpha} \right)^{\frac{1}{p'}}\\ &\le& \left(\int_{\mathbb{R}_+^2}(\mathcal{M}_{u,\alpha,\gamma}^{d,\beta}(gu^{-1}(z))^{p'}u(z)\mathrm{d}V_{\alpha}(z) \right)^{\frac{1}{p'}}\\ &\le& C \left(\int_{\mathbb{R}_+^2}(gu^{-1})^{q'}u\mathrm{d}V_{\alpha} \right)^{\frac{1}{q'}}= C\left(\int_{\mathbb{R}_+^2}(g\omega^{-1})^{q'}\mathrm{d}V_{\alpha} \right)^{\frac{1}{q'}}.
\Eeas
Hence,
\begin{eqnarray*}
L  &:=& \langle T_{\alpha,\gamma}f,g\rangle_{\alpha}\\ 
& \le & C[\omega]^{1+\frac{p'}{p} + \frac{q}{p'}}_{B_{p,q,\alpha}}\left(\int_{\mathbb{R}_+^2}(f\omega)^p\mathrm{d}V_{\alpha} \right)^{\frac{1}{p}} \left(\int_{\mathbb{R}_+^2}(g\omega^{-1})^{q'}\mathrm{d}V_{\alpha} \right)^{\frac{1}{q'}}.
\end{eqnarray*}
That is $$\langle T_{\alpha,\gamma}f,g\rangle_{\alpha} \le C[\omega]^{1+\frac{p'}{p} + \frac{q}{p'}}_{B_{p,q,\alpha}}\|f\omega\|_{p,\alpha}\|g\omega^{-1}\|_{q',\alpha}.$$ Hence taking the supremum over all $g\in L^{q'}(\omega^{-q'}\mathrm{d}V_{\alpha})$ with $\|g\omega^{-1}\|_{q',\alpha} = 1,$ we obtain $$\|(T_{\alpha,\gamma}f)\omega\|_{q,\alpha}\le C[\omega]^{1+\frac{p'}{p} + \frac{q}{p'}}_{B_{p,q,\alpha}}\|f\omega\|_{p,\alpha}.$$ The proof of the sufficiency is complete.
\vskip .2cm
To prove that the condition $\omega\in B_{p,q,\alpha}$ is necessary, recall that for any $f>0$, $$\mathcal{M}_{\alpha,\gamma}f(z)\le T_{\alpha,\gamma}f(z),\,\,\,\forall z\in \mathbb{R}_+^2.$$
Hence the boundedness of $T_{\alpha,\gamma}$ from $L^p(\omega^p\mathrm{d}V_{\alpha})$ to $L^q(\omega^q\mathrm{d}V_{\alpha})$ implies the boundedness of the maximal function $\mathcal{M}_{\alpha,\gamma}$ from $L^p(\omega^p\mathrm{d}V_{\alpha})$ to $L^q(\omega^q\mathrm{d}V_{\alpha})$. In particular, it implies that $\mathcal{M}_{\alpha,\gamma}$ is bounded from $L^p(\omega^p\mathrm{d}V_{\alpha})$ to $L^{q,\infty}(\omega^q\mathrm{d}V_{\alpha})$ which by Lemma \ref{lem:weakfrac} implies that $\omega\in B_{p,q,\alpha}$. The proof is complete.
\end{proof}
\vskip .3cm
\begin{proof}[Proof of Theorem \ref{thm:main2}]
Recall that for any $f>0,$ $$T_{\alpha,\gamma}f \lesssim\sum_{\beta\in \{0,\frac 13\}}Q_{\alpha,\gamma}^\beta f$$ where $Q_{\alpha,\gamma}^\beta $ is given by (\ref{eq:dyadicoperator}). The boundedness of $T_{\alpha,\gamma}$ then follows from the boundedness of $$\mathcal{Q}_{\alpha,\gamma}^{\beta}:L^{p_0}(\omega^{p_0}dV_{\alpha})\longrightarrow L^{q_0}(\omega^{q_0}dV_{\alpha}).$$ We observe that the latter is equivalent to the boundedness of $$\mathcal{Q}_{\alpha,\gamma}^{\beta}(\sigma\cdot):L^{p_0}(\sigma\mathrm{d}V_{\alpha})\longrightarrow L^{q_0}(u\mathrm{d}V_{\alpha})$$ where $\sigma=\omega^{-p_0'}$ and $u = \omega^{q_0}.$ For any $0<f\in L^{p_0}(\sigma\mathrm{d}V_{\alpha})$ and any $0<g\in L^{q_0'}(u\mathrm{d}V_{\alpha})$, we only have to estimate the quantity $\langle \mathcal{Q}_{\alpha}^{\beta}(\sigma f), gu\rangle_{\alpha}$. 

We have
\begin{eqnarray*}
L & := & \langle \mathcal{Q}_{\alpha}^{\beta}(\sigma f), gu\rangle_{\alpha}\\
& = & \sum_{I\in\mathcal{D}^\beta}\langle\sigma f, \frac{1_{Q_I}}{|I|^{2+\alpha -\gamma}}\rangle_{\alpha}\langle gu, 1_{Q_I}\rangle_{\alpha}\\
& = & \sum_{I\in\mathcal{D}^\beta}\frac{1}{|Q_I|_{\alpha}^{1-\frac{\gamma}{2+\alpha}}}\langle\sigma f, 1_{Q_I}\rangle_{\alpha}\langle gu, 1_{Q_I}\rangle_{\alpha}\\
& = & \sum_{I\in\mathcal{D}^\beta}\frac{|Q_I|_{u,\alpha}}{|Q_I|_{\alpha}}\left(\frac{|Q_I|_{\sigma,\alpha}}{|Q_I|_{\alpha}}\right)^{1-\frac{\gamma}{2+\alpha}}\langle\sigma f, \frac{1_{Q_I}}{|Q_I|_{\sigma,\alpha}^{1-\frac{\gamma}{2+\alpha}}}\rangle_{\alpha}\langle gu, \frac{1_{Q_I}}{|Q_I|_{u,\alpha}}\rangle_{\alpha}\times|Q_I|_{\alpha}\\
& \le & [\omega]_{B_{p_0,q_0,\alpha}}^{q_0}\sum_{I\in\mathcal{D}^\beta}|Q_I|_{\alpha}\left(\frac{1}{|Q_I|_{\sigma,\alpha}^{1-\frac{\gamma}{2+\alpha}}}\int_{Q_I}f\sigma dV_{\alpha}\right)\left(\frac{1}{|Q_I|_{\sigma,\alpha}}\int_{Q_I}gu dV_{\alpha}\right).
\end{eqnarray*}
Observe that $$u^{\frac{1}{1-\frac{\gamma}{2+\alpha}}}\sigma = 1 = u^{\frac{1}{q_0}\cdot\frac{1}{1-\frac{\gamma}{2+\alpha}}}\sigma^{\frac{1}{q_0}}= u^{\frac{1}{q_0'}}\sigma^{\frac{1}{q_0}}$$ and so $$|Q_I|_{\alpha}\simeq |T_I|_{\alpha}=\int_{T_I}u^{\frac{1}{q_0'}}\sigma^{\frac{1}{q_0}}dV_{\alpha}\le |T_I|^{\frac{1}{q_0}}_{\sigma,\alpha}|T_I|^{\frac{1}{q_0'}}_{u,\alpha}.$$ It follows using the H\"older's inequality and Lemma \ref{fracmax} that
\begin{eqnarray*}
L & := & \langle \mathcal{Q}_{\alpha}^{\beta}(\sigma f), gu\rangle_{\alpha}\\
& \le & [\omega]_{B_{p_0,q_0,\alpha}}^{q_0}\sum_{I\in\mathcal{D}^\beta}|T_I|_{\sigma,\alpha}^{\frac{1}{q_0}}\left(\frac{1}{|Q_I|_{\sigma,\alpha}^{1-\frac{\gamma}{2+\alpha}}}\int_{Q_I}f\sigma dV_{\alpha}\right)\times\\ && |T_I|_{u,\alpha}^{\frac{1}{q_0'}}\left(\frac{1}{|Q_I|_{u,\alpha}}\int_{Q_I}gu dV_{\alpha}\right)\\
& \le &  [\omega]_{B_{p_0,q_0,\alpha}}^{q_0}\left(\sum_{I\in\mathcal{D}^\beta}|T_I|_{\sigma,\alpha}\left(\frac{1}{|Q_I|_{\sigma,\alpha}^{1-\frac{\gamma}{2+\alpha}}}\int_{Q_I}f\sigma dV_{\alpha}\right)^{q_0} \right)^{\frac{1}{q_0}}\times \\
& & \qquad\left(\sum_{I\in\mathcal{D}^\beta}|T_I|_{u,\alpha}\left(\frac{1}{|Q_I|_{u,\alpha}}\int_{Q_I}gu dV_{\alpha}\right)^{q_0'} \right)^{\frac{1}{q_0'}}\\
& \le & [\omega]_{B_{p_0,q_0,\alpha}}^{q_0}\|\mathcal{M}_{\sigma,\alpha,\gamma}^{d,\beta}f\|_{q_0,\sigma,\alpha} \|\mathcal{M}_{u,\alpha}^{d,\beta}g\|_{q_0',u,\alpha}\\
& \le & C_{\alpha,\gamma}[\omega]_{B_{p_0,q_0,\alpha}}^{q_0}\|f\|_{p_0,\sigma,\alpha}\|g\|_{q_0',u,\alpha}.
\end{eqnarray*}
The proof is complete.
\end{proof}
\vskip .3cm
The author would like to thank the anonymous referees for carefully reading the manuscript and making suggestions that improved the presentation of the paper.

\end{document}